\documentclass[11pt,bezier]{article}
 \usepackage{amsmath,amssymb,amsfonts,euscript,graphicx,amsthm}
                        \usepackage[all]{xy}
                        \usepackage{multicol}
                        \usepackage{enumerate}
                          \usepackage{tikz}

  \evensidemargin =   -  3 cm \topmargin = -0.5 cm
  \newtheorem{The}{Theorem}[section]
  \newtheorem{Pro}[The]{Proposition}
  \newtheorem{Lem}[The]{Lemma}
  \newtheorem{Cor}[The]{Corollary}
  
  \newtheorem{Defs}[The]{Definitions}
  \newtheorem{Rem}[The]{Remark}

    \let\oldproofname=\proofname
   \renewcommand{\proofname}{\textit{\rm\bf\oldproofname}}

\title{\LARGE\bf Several Characterizations of Left K{\"o}the Rings \thanks {The research of the second  author was in part supported by two grants from IPM (No.1400130213), \indent\indent   and (No.1401130213). This research is partially carried out in the IPM-Isfahan Branch.}
\thanks {2020{ \it Mathematics Subject Classification}. Primary 16D70, 16G60, 16D90; Secondary 16D10,  16P20. }}
\author{{\bf {\bf Shadi Asgari}$^{{\rm b}}$, Mahmood Behboodi$^{{\rm a,b,}}$\thanks {Corresponding author.}\hspace{1mm} and {\bf Somayeh Khedrizadeh}$^{{\rm a}}$}  \\
{\small{ $^{{\rm a}}$Department of Mathematical Sciences, Isfahan University of Technology}}\vspace{-1mm}\\
{\small{ P.O.Box :  84156-83111,   Isfahan,   Iran}}\\
{\small{ $^{{\rm b}}$School of Mathematics, Institute for Research in Fundamental Sciences
(IPM)}}\vspace{-1mm}\\ {\small{ P.O.Box : 19395-5746, Tehran, Iran}}\vspace{-1mm}\\
   {\small{sh{\_}asgari@ipm.ir}}\vspace{-1mm}\\
   {\small{mbehbood@iut.ac.ir}}\vspace{-1mm}\\
   {\small{s.khedrizadeh@math.iut.ac.ir}}}
  \date{}
  
\begin{document}
  \maketitle
 \vspace*{-0.5cm}
 \begin{abstract}
 \small{
\noindent We study the classical K{\"o}the's problem, concerning the structure of non-commutative rings with the property that:  ``every left  module is a direct sum of cyclic modules".  In 1934,  K{\"o}the  showed that left modules over Artinian principal ideal rings are direct sums of cyclic modules. A ring $R$ is called a {\it left   K{\"o}the ring}  if every left   $R$-module is a direct sum of cyclic $R$-modules. In 1951, Cohen and Kaplansky  proved that all commutative K{\"o}the rings are Artinian principal ideal rings. During the years 1962 to 1965,  Kawada solved the K{\"o}the's problem for basic fnite-dimensional algebras: Kawada's theorem characterizes completely those finite-dimensional algebras for which any indecomposable module has square-free socle and square-free top, and describes the possible indecomposable modules. But, so far,  the K{\"o}the's problem is open in the non-commutative setting. In this paper,  we break the  class of left K{\"o}the rings  into three categories of nested:  {\it  left K{\"o}the rings},  {\it  strongly left K{\"o}the rings} and   {\it  very strongly left K{\"o}the rings}, and then, we solve the K{\"o}the's problem by giving several characterizations of these rings in terms of describing the indecomposable modules. Finally, we give a new generalization of K{\"o}the-Cohen-Kaplansky theorem.}
      \end{abstract}
      
 \section{\bf Introduction}
An old problem in noncommutative ring theory is to determine rings $R$ whose left modules are direct sums of cyclic modules (see [17, Question 15.8] and [27, Appendix B, Problem 2.48]). In 1934,  K{\"o}the  [22]  showed that over an Artinian principal ideal ring, each module is a direct sum of cyclic modules. A ring for which any left (resp., right) module is a direct sum of cyclic modules, is now called a {\it left}  (resp., {\it right})  {\it K{\"o}the  ring}, and characterizing such rings is called the K{\"o}the's problem. In 1941, Nakayama [25, 26] introduced the notion of generalized uniserial rings as a generalization of Artinian principal ideal  rings, and proved that generalized uniserial rings are  K{\"o}the rings (a ring $R$  is called a {\it generalized uniserial ring}, if $R$ has a unit element and every left ideal $Re$ as well as every right ideal $eR$  generated by a primitive idempotent element $e$ possesses only one composition series). However, as shown by Nakayama, the rings of this type are not general enough for solving the K{\"o}the's problem (see  [26, Page 289]).  In 1951, Cohen and Kaplansky [9] proved that all commutative K{\"o}the rings are Artinian principal ideal rings. Thus,  by combining the above results it is obtained that:

  \begin{The}
  {\rm (K{\"o}the,  [9])}. An Artinian principal ideal ring is a K{\"o}the ring.
   \end{The}

   \begin{The}
  {\rm (K{\"o}the-Cohen-Kaplansky, [9, 22])}.   A commutative ring $R$ is a K{\"o}the ring if and only if $R$ is an Artinian principal ideal ring.
   \end{The}

\noindent However, a left Artinian principal left ideal ring $R$ need not be a left  K{\"o}the ring, even if $R$ is a local ring (see [11, page 212, Remark (2)]). During the years 1962 to 1965,  Kawada [19, 20, 21] solved the K{\"o}the's problem for basic finite-dimensional algebras. Kawada's theorem characterizes completely those finite-dimensional algebras for which any indecomposable module has square-free socle and square-free top (called Kawada algebras), and describes the possible indecomposable modules. In fact, Kawada's theorem, contains a set of 19 conditions which characterize Kawada algebras, as well as the list of the possible indecomposable modules (see [18], for more ditals of 19 conditions). This seems to be the most elaborate result of the classical representation theory. Faith [11] characterized all commutative rings whose proper factor rings are K{\"o}the rings. Behboodi et al. [7] showed that  if $R$ is an Abelian left K{\"o}the ring, then $R$ is an Artinian principal right ideal ring (a ring in which all idempotents are central is called {\it Abelian}). So, it is obtained that

\begin{The}\label{abelian Kothe rings}
{\rm ([7])}. An Abelian ring $R$ is a K{\"o}the ring if and only if $R$ is an Artinian principal ideal ring. 
\end{The}

\noindent This generalizes K{\"o}the-Cohen-Kaplansky theorem. 
In this paper, we solve the K{\"o}the's problem, for a ring without any conditions, by giving several characterizations in terms of describing the indecomposable modules. Also, we give a generalization of Theorem \ref{abelian Kothe rings}.

Throughout this paper, all rings have an identity element and all modules are unital. A K{\"o}the (resp. Noetherian, Artinian) ring is a ring which is both a left and right K{\"o}the (resp. Noetherian, Artinian) ring. A principal ideal ring is a ring which is both a left and a right principal ideal ring. For a module $M$, the {\it top} of $M$, denoted by $top(M)$, is the factor module $M/Rad(M)$ where $Rad(M)$ denotes the radical of $M$. The module $M$ is called {\it square-free} if it does not contain a direct sum of two nonzero isomorphic submodules. In Section 2, we provide several characterizations for a left K{\"o}the ring. Among these, it is proved that $R$ is a left K{\"o}the ring if and only if $R$ is of finite representation type and every (finitely generated) indecomposable left $R$-module has a cyclic top (see Theorem \ref{left Kothe}). In the sequel, we introduce and study the following concepts. Using them, we obtain more characterizations of K{\"o}the rings. 
 
\begin{Defs}
{\rm We say that a ring $R$  is a: \vspace*{1mm} \\
{\it strongly left}  (resp., {\it right})  {\it K{\"o}the ring}    if every  nonzero left  (resp.,  right)  $R$-module is a  direct sum of  modules  with nonzero  square-free  cyclic top. \vspace*{1mm} \\
{\it very  strongly  left}   (resp., {\it right}) {\it  K{\"o}the ring}  if every  nonzero left (resp.,  right)   $R$-module is a  direct sum of   modules with  simple top. \vspace*{1mm} \\
{\it strongly  K{\"o}the ring}  if $R$ is both  a  strongly left and  a  strongly right  K{\"o}the ring. \vspace*{1mm} \\
{\it very strongly  K{\"o}the ring}  if $R$ is both  a   very  strongly left and   right K{\"o}the  ring.}
\end{Defs}

\noindent It will be shown that these types of rings are left K{\"o}the (see Theorem \ref{left Kothe}), so we have the following inclusion relationships: \vspace*{1mm} \\
\indent Very strongly left K{\"o}the rings $\subsetneqq$ Strongly left K{\"o}the rings $\subsetneqq$ Left K{\"o}the rings. \vspace*{1mm} \\
Several equivalent conditions for the concept of a strongly left K{\"o}the ring are given in Section 3 (see Theorem \ref{strongly left Kothe}). We show that the notions of a left K{\"o}the ring and a strongly left K{\"o}the ring are the same for left quasi-duo rings (see Theorem \ref{quasi duo left Kothe}). In Section 4, we characterize very strongly left K{\"o}the rings and prove that very strongly K{\"o}the rings are exactly Artinian serial rings (see Theorems \ref{left local type} and \ref{co-Kothe}). It is shown that for Abelian rings, the notions of a K{\"o}the ring, a strongly K{\"o}the ring, a very strongly K{\"o}the ring and an Artinian principal ideal ring coincide (see Proposition \ref{Cor prime ideals commutes-Abelian}). Therefore, an Abelian ring $R$ is Artinian serial if and only if $R$ is an Artinian principal ideal ring. This generalizes a well-known result of Asano stating that a commutative ring $R$ is Artinian serial if and only if $R$ is an Artinian principal ideal ring (see [2, 3]). Moreover, the equivalency of the notions of a very strongly K{\"o}the ring and an Artinian serial ring is a generalization of Theorem \ref{abelian Kothe rings}, and so it is a new generalization of K{\"o}the-Cohen-Kaplansky theorem. 

\newpage
 \section{\bf  Charactrizations of left K{\"o}the rings}

In this section, we solve the left K{\"o}the's problem by giving several characterizations in terms of describing the indecomposable modules. Recall that a submodule $N$  of an $R$-module $M$ is called an {\it $RD$-pure submodule} if  $N\cap rM=rN$ for all
$r\in R$. Every direct summand of $M$ is a $RD$-pure submodule (we note that in [8], Chase has used the term ``pure submodule" instead of ``RD-pure submodule").

The following lemmas  play  an important role in the sequel. 

\begin{Lem}\label{chase}\textup {(Chase [8, Theorem 3.1])}
Let $R$ be a ring, and $A$  an infinite set of cardinality $\zeta$ where
$\zeta\geq card(R)$. Set $M =\prod_{\alpha\in A}R^{(\alpha)}$, where $R^{(\alpha)}\cong R$ is a left $R$-module. Suppose
that $M$ is an $RD$-pure submodule of a left $R$-module of the form $N=\bigoplus_\beta N_\beta$, where
each $N_\beta$ is generated by a subset of cardinality less than or equal to $\zeta$. Then $R$
must satisfy the descending chain condition (DCC) on principal right ideals. \\
\end{Lem}

\begin{Lem}\label{Zimmermann} \textup {(B. Zimmermann-Huisgen-W. Zimmermann [34, Page 2])}
For a ring $R$ the following statments are equivalent. \vspace*{0.2cm} \\
\noindent {{\rm (1)}} Each left $R$-module is a direct sum of finitely generated modules.\\
\noindent {{\rm (2)}} There exists a cardinal number $\aleph$ such that each left  $R$-module is a direct sum of  \indent   ${\aleph}{-\rm generated ~modules.}$\\
\noindent {{\rm (3)}} Each left  $R$-module is a direct sum of indecomposable modules (R is left pure semisim-\indent ple). 
\end{Lem}

Let $M$ be a left $R$-module.  By [32, Proposition 21.6 (4)], if $Rad(M)$ is a superfluous submodule of $M$,  then $M$ is finitely generated if and only if $top(M)$ is finitely generated. Similar to this, we have the following lemma. 

\begin{Lem}\label{square-free top} 
Let $M$ be a  left $R$-module such that $Rad(M)$ is a superfluous submodule of $M$. Then $M$ is cyclic if and only if $top(M)$  is  cyclic.
\end{Lem}

\begin{proof}
Let $top(M)=M/Rad(M)$ be cyclic with the generating element $x + Rad(M)$, where $x\in M$. Then for each $m \in M$,  $m + Rad(M) = rx + Rad(M) $ for some $r \in R$. Hence, $m - rx \in Rad(M)$ and so $m = (m - rx) + rx\in
Rad(M) + Rx$. Thus, $Rad(M) + Rx = M$  and so $M = Rx$, since $Rad(M)$ is superfluous
in $M$.
\end{proof}

A ring $R$ is called {\it semiperfect}  if  $R/J$  is left semisimple and idempotents in $R/J$  can be
lifted to $R$.  Local rings and left (or right) Artinian rings are semiperfect. A ring $R$ is called left (right) {\it perfect} in case each of its left (right) modules has a projective cover. Left perfect rings and right perfect rings are both semiperfect. It is worthy of note that in a semiperfect ring,  the radical is the unique largest ideal containing no nonzero idempotents (see  [1, (15.12)]). Moreover, a semiperfect ring $R$ decomposes as a left  $R$-module into cyclic modules:  $R= {(Re_1)}^{(t_1)}  \oplus\cdots\oplus {(Re_n)}^{(t_n)}$. If $A= Re_1  \oplus\cdots\oplus Re_n$, then the category of left $R$-modules and left $A$-modules are Morita equivalent. The semi-perfect ring $R$ is said to be {\it basic}  if $t_1 =  \cdots = t_n = 1$, i.e., there are no
isomorphic modules in the decomposition of $R$ into cyclic modules. In fact, a semiperfect ring $R$  is  {\it basic} if the quotient ring $R/J$ is a direct sum of division rings. Let $M$ and $N$ be $R$-modules.  A monomorphism   $ N \rightarrow M$   is called an  {\it embedding}  of  $N$  to  $M$,   and  we denote  it by $N\hookrightarrow M$). The following proposition,  play a key role  in this paper.

\vspace*{1mm}

\begin{Pro}\label{Projective cover}
Let $R$ be a semiperfect ring and $M$ be a  finitely generated left $R$-module with the projective cover $P(M)$. Then 
$R\cong \oplus_{i=1}^n(Re_i)^{(t_i)}$,  where $n, t_i\in\Bbb{N}$ and $\{e_1,\ldots, e_n\}$ is a basic  set of idempotents     of $R,$ $R/J \cong \oplus_{i=1}^n (Re_i/Je_i)^{(t_i)}$,  $top(M)\cong \oplus_{i=1}^n (Re_{i}/Je_i)^{(s_i)}$,     for  some
 $s_1,\ldots,s_n\in\Bbb{N}\cup \{0\}$, and   $P(M) \cong \oplus_{i=1}^n(Re_i)^{(s_i)}$. Consequently, \vspace*{0.1cm} \\
 \indent  $(i)$  the followings are equivalent for $M$. \vspace*{0.1cm} \\
 \indent\indent{\rm (1)}  $M$ is  a cyclic  $R$-module.\\
 \indent\indent {\rm (2)}   $top(M)$ is a  cyclic module.\\
 \indent\indent{\rm (3)}   $M\cong Re/Ie$ for some idempotent $e\in R$ and  for some left ideal  $I\subseteq J$.\\
 \indent\indent {\rm (4)}  $Re$ is equal to the projective cover of $M$ for some  idempotent $e\in R$. \\
 \indent\indent {\rm (5)} By the notations above, we have   $s_i\leq t_i$ for all $i=1,\ldots, n$.\\
 \indent\indent {\rm (6)}   $top(M)\hookrightarrow R/J$.\vspace*{0.1cm} \\
 \indent\indent  In particular,  $top(M)$  is square-free  if and  only if $s_i \in \{0,1\}$ for all $i=1,\ldots, n$. \vspace*{0.1cm} \\
 \indent  $(ii)$   The followings are equivalent for $M$. \vspace*{0.1cm} \\
 \indent\indent{\rm (1)}  $M$ is a local module.\\
 \indent\indent{\rm (2)}  $top(M)$ is  a simple  module.\\
  \indent\indent{\rm (3)}  $P(M)$ is an indecomposable module.\\
 \indent\indent{\rm (4)}  $P(M)=Re$ for some local idempotent $e\in R$.\\
 \indent\indent{\rm (5)}  $P(M)$ is  a projective cover of a simple  module. 
 \end{Pro}
 
\begin{proof}  The first statement is  by [1, Corollary 15.18, Theorems 27.6,  27.13   and Proposition 27.10].

(i) $(1) \Leftrightarrow (2).$  It is by Lemma \ref{square-free top}, since $M$ is finitely generated, and so by   [32, Proposition 21.6 (4)], $Rad(M)$ is superfluous in $M$ .

    (i) $(1) \Leftrightarrow (3).$  Assume that $M$ is a cyclic  left $R$-module. Since $R$ is  a semiperfect ring,  $M$ has a   projective cover by [1, Corollary  27.6].  Also  by [1, Lemma 27.3],
   $M\cong Re/Ie$ for some idempotent $e\in R$ and some left ideal $I \subseteq J$,  and  the natural map
$Re\rightarrow   Re/Ie\rightarrow  0$  is a projective cover  of $M$.

(i) $(3) \Rightarrow (4).$  By [1, Lemma 27.3], for some idempotent $e$ of $R$ and  some left ideal $I \subseteq J$,  the natural map $Re\rightarrow   Re/Ie\rightarrow  0$  is a projective cover  of $M$.

 (i) $(4) \Rightarrow (1).$ Since  $e$ is an  idempotent in $R$, so $Re$ is a direct summand of $_RR$. Then $Re = P(M)$, and hence $M$ is cyclic.

 (i)  $(1) \Rightarrow (6).$  If $M$ is cyclic, then by Part (a) above,  $top(M)=M/Rad(M)$ is a cyclic $R/J$-module,    and it follows that  $top(M)$ is isomorphic to a direct summand of   $R/J$.

 (i)  $(6) \Rightarrow (1).$ It is clear.

 (i)  $(2) \Rightarrow (5).$  It is well known that, $J(R) top(M) =0$.  If $top(M)$ is cyclic, then $top(M)$ isomorphic to a summand of $R/J(R)$ and hence it can be embedded in $R/J(R)$. So,  $s_i\leq t_i$ for all $i=1,\ldots, n.$
 
 (i) $(5) \Rightarrow (2)$. If $s_i\leq t_i$ for all $i=1,\ldots, n$, then $top(M)$  is isomorphic to a direct summand of $R/J$ and hence $top(M)$ is a cyclic module.

  (i)  $(5) \Leftrightarrow (6).$ It is clear.

  (ii).  $(1) \Rightarrow (2).$ It is clear.

  (ii). $(2) \Rightarrow (1).$ Since $top(M)$ is simple, $Rad(M)$ is the unique maximal submodule of $M$. Finally, since $M$ is nonzero and finitely generated, every proper submodule of $M$ is contained in a maximal submodule, and therefore  $Rad(M)$ is the greatest proper submodule of $M$.

The rest of the proof (ii) is by [16, Theorem 1.9.4] and  [16, Proposition 1.9.6.].
\end{proof}

\begin{Cor}\label{top square-free is cyclic}
Let $R$ be a semiperfect ring and $M$ be a  finitely generated left $R$-module with the projective cover $P(M)$.
 If $top(M)$ is  square-free, then $P(M)$ (and so $M$)  is   cyclic. 
\end{Cor}
 
\begin{proof}
If  $top(M)$ is  square-free, then  by Proposition \ref{Projective cover}, we have;  $$top(M) \cong (Re_{i_1}/Je_{i_1})\oplus\cdots\oplus (Re_{i_k}/Je_{i_k}),$$
where  $e_{i_1}, \ldots, e_{i_k}$ are distinct elements of $\{e_1,\ldots, e_n\}$. Also,  by Proposition \ref{Projective cover},    there exists a projective cover $f: P=Re_{i_1}\oplus \cdots\oplus Re_{i_k}\rightarrow M$. Clearly  $P$ is a direct summand of $R$, and   so
$P$ and $M\cong P/Ker(f)$ are  cyclic modules.
\end{proof}

Recall that a ring $R$ is of {\it finite representation type}   (or  {\it finite type}) if $R$ is Artinian and there is a finite number of the isomorphism classes of finitely generated indecomposable left (and right) $R$-modules.  The ring $R$  is said to be of {\it left bounded (representation) type}, if it is  left Artinian and there is a finite upper bound for the lengths of the finitely generated indecomposable modules in $R$-Mod.  By [32, Proposition 54.3], $R$  is of left bounded representation type if and only if $R$ is of finite representation type. Moreover, A subset $X$ of a ring $R$ is called left $T$-nilpotent if, for every sequence $x_1, x_2, \ldots$ of elements in $X$, there is an  $n \in \mathbb{N}$ with $x_1x_2 \cdots x_n = 0$.
Similarly, a  right $T$-nilpotent  subset  of $R$ is defined. The importance of the concept of $T$-nilpotence is due to the fact that the
radical $J(R)$ of a ring $R$ is left $T$-nilpotent precisely when the ``Nakayama's Lemma"  [1, Corollary 15.13] holds for all left modules, finitely generated or not.

We are now in a position to provide the following properties  and characterizations of left K{\"o}the rings.

\begin{The}\label{left Kothe}
The following statements are equivalent for any  ring $R$. \\
\indent {{\rm (1)}}   $R$ is a left  K{\"o}the  ring.\\
\indent {{\rm (2)}}  Every  nonzero left $R$-module is a direct sum of  modules with nonzero socle and  \indent\indent cyclic top.\\
\indent {{\rm (3)}} $R$  is  left  perfect  and every left  $R$-module is a direct sum of  modules with  cyclic top.\\
\indent {{\rm (4)}}  $R$ is  left Artinian  and  every   left $R$-module  is a direct $sum~of~modules~with~cyclic~top.$\\
\indent {{\rm (5)}} $R$ is  left pure semisimple  and  every   left $R$-module  is a direct sum of modules with \indent\indent cyclic top.\\
\indent {{\rm (6)}}  $R$ is  of finite representation type and  every  $($finitely generated$)$ indecomposable left  \indent\indent $R$-module  has a cyclic  top.\\
 \indent {{\rm (7)}}   $R$ is of finite representation type  and  $top(U)\hookrightarrow R/J$ for each  indecomposable left \indent\indent  $R$-modules $U$.\\
\indent {{\rm (8)}}    $R\cong \oplus_{i=1}^n(Re_i)^{(t_i)}$, where $n,~t_i\in\Bbb{N}$,  $i=1,\ldots,n$,  $\{Re_1,\ldots,Re_n\}$   is a   complete  \indent\indent  set of representatives of the isomorphic classes of
indecomposable projective left  $R$-   \indent\indent  modules,  and  each indecomposable  left   $R$-module is  isomorphic to  $Re/Ie$  for  some  \indent\indent  idempotent $e\in R$
 and  for  some   left ideal $I\subseteq J$.\\
 \indent {{\rm (9)}}  $R$ is of finite representation type and  $R\cong \oplus_{i=1}^n(Re_i)^{(t_i)}$, where $n, t_1, \ldots, t_n \in\Bbb{N}$, \indent\indent  $\{Re_1,\ldots,Re_n\}$   is a   complete  set of representatives of the isomorphic classes   of    \indent\indent indecomposable projective left  $R$-modules and   $top(U) \cong \oplus_{i=1}^n(Re_i/Je_i)^{(u_i)},$   for some
 \indent\indent $u_i\in\Bbb{N}\cup \{0\}$,  where   $u_i\leq t_i$ for all $i=1,\ldots,n $ for each  indecomposable left $R$- \indent\indent module $U$.
\end{The}

\begin{proof}
(1) $\Rightarrow$ (2). By hypothesis, every $R$-module is a direct sum of cyclic modules. Clearly, the top of every cyclic module is cyclic. Also, $R$ is left Artinian by [8, Theorem 4.4], and so every cyclic module has a nonzero socle.

(2) $\Rightarrow$ (3).  Set $\bar{R}=R/J$ where $J$ is the Jacobson radical of $R$, and $M =\prod_{\alpha\in A}\bar{R}^{(\alpha)}$ where   $A$ is  an infinite set of cardinality $\zeta\geq card(\bar{R})$, and  $\bar{R}^{(\alpha)}\cong \bar{R}$ is a left $\bar{R}$-module. Clearly,  every left $\bar{R}$-module is  also a direct sum of  modules with cyclic top. Thus,  $M =\prod_{\alpha\in A}{\bar{R}}^{(\alpha)}=\bigoplus_{\lambda\in\Lambda}N_\lambda$, where $\Lambda$ is an index set and each  $N_\lambda$ has a cyclic top.
Since  $J(\bar{R})=0$, one can easily see that $Rad(M)=0$ (see for instance,  [1, Page 174, Exercises 3]),  and hence by [32, Proposition 21.6 (5)], we conclude that $Rad(N_\lambda)=0$ for each $\lambda\in\Lambda$. It follows that  $N_\lambda$ is cyclic for each $\lambda\in\Lambda$. Thus,  $M$ is a direct sum of cyclic  left $\bar{R}$-modules. Since the cardinality  of each cyclic left $\bar{R}$-module is  at most $card(\bar{R})\leq \zeta$, by  Lemma \ref{chase},  $\bar{R}$ must satisfy the descending chain condition on principal right  ideals. By Bass [6, Theorem P],  the ring $\bar{R}$ is a semisimple ring, i.e., $R/J$ is a semisimple ring. Now, since by our hypothesis, every nonzero left $R$-module has a minimal submodule, we conclude that $J$ is $T$-nilpotent by [1, Remark 28.5],  and hence by  [1, Theorem 28.4],  $R$ is a  left perfect ring.

(3) $\Rightarrow$ (1).
Since $R$ is a  left perfect ring,  by  [1, Theorem 28.4],  every nonzero left $R$-module contains a maximal
submodule.  It follows that for each left $R$-module $M$, every proper submodule of $M$ is contained in a maximal submodule.  Now let  $M$ be a module with cyclic top. Then by [32, Proposition   21.16 (3)],  $Rad(M)$ is a superfluous submodule of $M$, and so by Lemma \ref{square-free top},  $M$ is cyclic. Thus,  every  left $R$-module is a direct sum of  cyclic modules, i.e., $R$ is a left  K{\"o}the  ring.

(1) $\Rightarrow$ (4).   Since $R$ is  a left K{\"o}the ring,  by [8, Theorem 4.4],  $R$ is  a left Artinian ring.  Also, it is clear that each  cyclic left $R$-module  has a cyclic top.

(4) $\Rightarrow$ (3). Since $R$ is  left  Artinian, $R$ is a left prefect ring by [1, Corollary 28.8].

(1) $\Rightarrow$ (5).   Since $R$ is  a left K{\"o}the ring, $R$ is a  left pure semisimple  ring by Lemma \ref{Zimmermann}. Now it is clear that each  cyclic left $R$-module  has a cyclic top.

(5) $\Rightarrow$ (6). Since $R$ is  a left pure semisimple,  by [8, Theorem 4.4],  $R$ is  a left Artinian ring and every finitely generated indecomposable left $R$-module is cyclic. Thus,  $R$ is of finite  length and $length(R)$  is a finite upper bound for the lengths of the finitely generated indecomposable left $R$-modules. i.e., $_RR$   is bounded. Now,  by [32, Proposition  54.3)], $R$ is of finite representation type.

(6) $\Leftrightarrow$ (7). It  is by Proposition \ref{Projective cover}(i).

(7) $\Rightarrow$ (8).  It is by Proposition \ref{Projective cover}(i) and [1, Proposition 27.10].

(8) $\Rightarrow$ (9).  It is by Proposition \ref{Projective cover}(i).

(9) $\Rightarrow$ (1). It is clear by Proposition \ref{Projective cover}(i).
\end{proof}

The  left K{\"o}the  property is not a Morita invariant property. In fact,   there exists a  ring $R$ and a positive integer  $n\geq 2$ such that the matrix  ring $M_n(R)$  is a left  K{\"o}the ring  but $R$ is not a left K{\"o}the ring (see [12, Proposition 4.7, Remark 4.8]). In the last result of this section we give a characterization of a ring $R$ for which any Morita equivalent to it is a left K{\"o}the ring:

\begin{Pro}\label{Kawada}
The following statements are equivalent for a ring $R$. \vspace*{0.2cm} \\
\indent {{\rm (1)}} Any ring Morita equivalent to $R$ is a left K{\"o}the ring.\\
\indent {{\rm (2)}} $R$ is a semiperfect ring and the basic ring $A$ of $R$ is a left K{\"o}the ring.
\end{Pro}
\begin{proof}
(1) $\Rightarrow$ (2). By hypothesis, $R$ is a left K{\"o}the ring, and so $R$ is left Artinian. Thus by [1, Proposition 27.14], the basic ring $A_0$ of $R$ is Morita equivalent to $R$. Hence by assumption, $A_0$ is a left K{\"o}the ring.

(2) $\Rightarrow$ (1). Let $T$ be a ring Morita equivalent to $R$. Then by [1, Proposition 27.14], $T$ is Morita equivalent to $A$, and hence by [12, Corollary 4.2], $T$ is a left K{\"o}the ring.
\end{proof}

\section{\bf Charactrizations of strongly left  K{\"o}the rings}
In this section, we study strongly left K{\"o}the rings. Also, we give more characterizations of a left K{\"o}the ring $R$, whenever $R$ is left quasi-duo. In the sequel, we denote by  $R$-Mod (resp., $R$-mod)  the category of left $R$-modules (resp., finitely generated left $R$-modules), and  by  Mod-$R$ (resp., mod-$R$)  the category of right $R$-modules (resp., finitely generated right $R$-modules). For an $R$-module $M$,  $E(M)$  denotes the {\it injective hull}  of $M$. Now, we explain some concepts regarding the functors rings  of the finitely generated modules of  $R$-Mod, which will be used extensively in this section. Let $\{V_\alpha\}_A$  be a family of finitely generated $R$-modules and $V = \bigoplus_AV_\alpha$. For any $N\in R$-Mod we define:
$$\widehat{H}om(V, N)=\{f\in Hom(V, N) ~|~ f(V_\alpha) = 0~{\rm for~ almost~ all~} \alpha\in A\}.$$
For $N=V$ , we write $\widehat{H}om(V, V )= \widehat{E}nd(V)$. Note that these constructions do not depend on the decomposition of $V$ (see [32, Chap. 10, Sec. 51] for more details).

\begin{Rem}\label{finite representation type}
{\rm Let $R$ be  a ring of  finite representation type, and $V:=V_1\oplus\cdots\oplus V_n $,  where   $\{V_1,\ldots,V_n\}$  is a complete set of representatives of the isomorphic classes of finitely generated indecomposable left $R$-modules. 
Clearly, the functor ring $T=\widehat{E}nd(V)$ is equal to  $ End_R(V)$,  and in this case, the ring  $T:= End_R(V)$ is called the {\it left Auslander ring}  of $R$. By [4, Proposition 3.6],  $T$  is an Artinian ring and so soc$(T_T)$ is an essential submodule  of $T_T$ and $T$ contains only finitely many non-isomorphic types of simple modules.  Also  by [32, Proposition 46.7],  the functor  ${\rm Hom}_R(V,-)$  establishes an equivalence between the category of left $R$-modules and the full subcategory of finitely generated  projective  left $T$-modules,  which preserves and reflects finitely generated  left $R$-modules and  finitely generated indecomposable left $R$-modules  correspond to finitely generated indecomposable projective $T$-modules (see [32, Proposition  51.7(5)]). In fact,
each finitely generated indecomposable projective $T$-module is  local and hence it  is the projective cover of  a simple $T$-module. Therefore ${\rm Hom}_R(V,-)$ yields a bijection between a minimal representing set of finitely generated, indecomposable left $R$-modules   and the set of projective covers of non-isomorphic simple left $T$-modules.}
\end{Rem}

Recall that an injective left $R$-module $Q$  is a {\it cogenerator} in the category of left $R$-modules  if and only if it
cogenerates every simple left $R$-module, or equivalently, $Q$ contains every simple left $R$-module as a submodule (up to isomorphism) (see [32, Proposition 16.5]
Also,  a cogenerator  $Q$ is called {\it minimal cogenerator}  if $Q\cong \oplus_{i\in I}E(S_i)$, where  $I$ is an index set and $\{S_i~|~i\in I\}$ is a complete set of representatives of the isomorphic classes of simple left $R$-modules.
Let $R$ and $S$ be two rings. An additive contravariant functor $F:R$-mod$\rightarrow$mod-$S$ is called {\it duality}  if it is an equivalence of categories (see [32, 19, Chap. 9]). Moreover, a ring $R$ is called a {\it left Morita ring}  if there is an injective cogenerator $_RU$  in $R$-Mod such that for $S={\rm End}(_RU)$, the module  $U_S$ is an injective cogenerator in Mod-$S$ and $R\cong {\rm End}(U_S)$. The following result summarizes all the facts presented above on rings of finite representation  type.

\begin{Pro}\label{equivalence}
Let $R$ be a ring of finite representation type,  $U:=U_1\oplus\cdots\oplus U_n $,  where   $\{U_1,\ldots,U_n\}$ is   a complete set of representatives of the isomorphic classes of finitely generated indecomposable left $R$-modules,  $Q\cong E(S_1)\oplus\cdots\oplus E(S_m)$, where  $\{S_i~|~1\leq i\leq m\}$ is a complete set of representatives of the isomorphic classes of simple left $R$-modules,  $T:=End_R(U)$ and $S= End_R(Q)$. Then \vspace*{0.2cm} \\
\indent  {\rm (a)}  The functor  ${\rm Hom}_R(U,-)$  establishes an equivalence between the category of left  \indent \indent $R$-modules and the full subcategory of finitely generated  projective  left $T$-modules.  \\
\indent  {\rm (b)}  The functor ${\rm Hom}_R(-,Q)$ : $R$-Mod$\rightarrow$  Mod-$S$ is a duality  with the inverse duality \indent \indent ${\rm Hom}_S(-,Q) : $Mod-$S\rightarrow  R$-Mod. \\
\indent  {\rm (c)}   $S$ is a right Artinian ring and  $\{Hom_R(U_1,Q),\ldots,Hom_R(U_n,Q)\}$ is a complete \indent \indent set of  representatives of the isomorphic classes of finitely generated indecomposable \indent \indent right  $S$-modules.  \\
\indent  {\rm (d)} The  left Auslander ring  $T$ of $R$ is isomorphic to the  right  Auslander ring  $T^\prime$ of  $S$ \indent\indent   (thus $R$  is left Morita to $S = End_R(Q)$).   \\
\indent  {\rm (e)}  If  $P$  is  an indecomposable projective   left $T$-module, then  $P\cong Hom_R(U,M)$  for  \indent  \indent some indecomposable left  $R$-module $M$.\\
\indent  {\rm (h)}  If  $M$  is a finitely generated (indecomposable)  left $R$-module with square-free  top,  \indent\indent   then   the right $S$-module    ${\rm Hom}_R(M,Q)$  has a square-free  socle.  \\
\indent  {\rm (i)}  If  $M$  is  an indecomposable  left $R$-module with square-free  top,   then  $M$ is cyclic with \indent\indent simple top and so  the right $S$-module   ${\rm Hom}_R(M,Q)$  has a simple socle.  \\
\indent  {\rm (j)}  If  $M$  is  an indecomposable  left $R$-module with simple  top,   then  the right $S$-module \indent\indent ${\rm Hom}_R(M,Q)$  has a simple socle.  \\
\indent  {\rm (m)}  For each finitely generated  left $R$-module $X$, as $S$-modules,  we have:
  $$soc(Hom_R(X,Q))\cong  Hom_R(top(X),Q).$$
\end{Pro}
\begin{proof}
(a). Since $R$ is finite of representation type, $U=\bigoplus_AU_\alpha$ is a generator in $R$-Mod by Remark \ref{finite representation type}  and  [32, Chap. 10, Sec. 52]. The rest of proof is obtained by [32, Page 507, Part(10)].

 (b) and (c). By [32, Proposition 47.15].

 (e). It is by [32, Chap. 10, Sec. 52].

 (d). See [32, 52.9, Exercises, Exercise 1].

 The facts (h), (i), (j) and  (m) are by (b), (c) and [32, Proposition 47.3] (we see that  the functor ${\rm Hom}_R(-,Q)$: $R$-Mod$\rightarrow$  Mod-$S$ is a duality with the inverse duality ${\rm Hom}_S(-,Q)$: Mod-$S\rightarrow  R$-Mod).
 Also by [1, Proposition 24.5] for each finitely generated left $R$-module $X$, the lattice of all $R$-submodules of $M$ and the lattice of all $S$-submodules of ${\rm Hom}_R(M,Q)$  are anti-isomorphic. Hence,  for each finitely generated left $R$-module $X$, ${\rm soc(Hom}_R(X,Q))\cong {\rm Hom}_R(top(X),Q)$, and  $top(Hom(X,Q)) = \frac{Hom(X,Q)}{Rad(Hom(X,Q))}\cong Hom(soc(X), Q))$ as $S$-modules (see [1, 24. Exercises, Exercises 6]).
\end{proof}

Recall that a semi-perfect ring $R$ is said to be a {\it left}  (resp., {\it right})  {\it $QF$-2 ring}  if every indecomposable projective left (resp., right) $R$-module has a simple essential socle (see [13]). This definition and the fact that simple modules are square-free motivated us to define the following concept.

\begin{Defs}
{\rm We say that a semi-perfect ring $R$  is a {\it generalized left} (resp., {\it right})  {\it $QF$-2  ring} if every indecomposable projective left (resp.,  right)  $R$-module has a square-free  essential  socle. The ring $R$  is called  a  {\it generalized }  {\it $QF$-2  ring} if $R$ is both a generalized left  and a generalized  right  $QF$-2  ring.}
\end{Defs}

We are now in a position to provide the following properties  and characterizations of  stongly left K{\"o}the rings.

\begin{The}\label{strongly left Kothe}
The following statements are equivalent for a ring $R$.  \vspace*{0.2cm} \\
\indent {{\rm (1)}}   $R$ is a stongly left  K{\"o}the  ring.\\
\indent {{\rm (2)}}  Every  nonzero left $R$-module is a direct sum of  modules with nonzero socle and  \indent\indent square-free cyclic top.\\
\indent {{\rm (3)}} $R$  is  left  perfect,   and every left  $R$-module is a direct sum of  modules with  square-free \indent\indent  top.\\
\indent {{\rm (4)}}  $R$ is  left Artinian  and  every   left $R$-module  is a direct sum of modules with square- \indent\indent free top.\\
\indent {{\rm (5)}} $R$ is  left pure semisimple,   and  every   left $R$-module  is a direct sum of modules with \indent\indent square-free top.\\
\indent {{\rm (6)}}  $R$ is  of finite representation type,  and  every  $($finitely generated$)$ indecomposable left  \indent\indent $R$-module  has a  square-free  top.\\
\indent {{\rm (7)}}  $R$ is  of finite  representation type,   and  for the left Morita dual ring of $R$, every  \indent\indent (finitely generated) indecompsable right module has a square-free socle.\\
\indent {{\rm (8)}}  $R$ is  of finite representation  type,   and the left  Auslander ring  of $R$  is a right gener- \indent\indent alized $QF$-2 ring.\\
 \indent {{\rm (9)}}   $R$ is of finite representation type,   and   for each  indecomposable left   $R$-module $U$,
 \indent\indent  $top(U)\hookrightarrow  (Re_1/Je_1)\oplus\cdots\oplus (Re_n/Je_n),$  where  $e_1,\cdots, e_n$ is a basic set of primitive \indent\indent idempotents for $R$.\\
 \indent {{\rm (10)}}  $R\cong \oplus_{i=1}^n(Re_i)^{(t_i)}$, where $n,~t_i\in\Bbb{N}$,  $\{Re_1,\cdots, Re_n\}$  is a   complete   set of   repre-  \indent\indent sentative of the isomorphic classes of
indecomposable projective left $R$-modules,   \indent\indent each  indecomposable left $R$-module has a square-free  top,  and  it is  isomorphic  to \indent\indent  $Re/Ie$    for some idempotent $e \in R$ and  a  left ideal $I\subseteq J$.
  \end{The}
\begin{proof}
(1) $\Rightarrow$ (2).  Let $0 \neq M = \bigoplus_{\Lambda}M_{\lambda}$, where $\Lambda$ is an index  set and each $M_{\lambda}$ is a left $R$-module with nonzero square-free cyclic top. Hence,  $Rad(M_{\lambda}) \neq M_{\lambda}$ for each $\lambda \in\Lambda$.   Since each  $M_{\lambda}$ has a cyclic top, $Rad(M) = \bigoplus_{\Lambda} Rad(M_{\lambda}) \neq \bigoplus_{\Lambda}M_{\lambda} = M$ by [32, Proposition 21.6 part (5)].  So every nonzero left $R$-module contains a maximal submodule.  It follows that for each left $R$-module $M$, every proper submodule of $M$ is contained in a maximal submodule.  Now let  $M$ be a module with square-free cyclic top. Then by [32, Proposition   21.16 (3)],  $Rad(M)$ is a superfluous submodule of $M$, and so by Lemma \ref{square-free top},  $M$ is cyclic. Thus,  every  left $R$-module is a direct sum of  cyclic modules, i.e., $R$ is a left  K{\"o}the  ring. Hence by [8, Theorem 4.4],  $R$ is  a left Artinian ring,  and so every  nonzero left $R$-module has a nonzero socle.

(2) $\Rightarrow$ (3). By Theorem \ref{left Kothe}, $R$ is  a left K{\"o}the ring. Thus by  [8, Theorem 4.4],  $R$ is  a left Artinian ring and  so  by [1, Corollary 28.8], $R$ is a left prefect ring.

(3) $\Rightarrow$ (5).
Since $R$ is a  left perfect ring, by Corollary \ref{top square-free is cyclic},  $R$ is a left K{\"o}the ring. So by Theorem \ref{left Kothe}, $R$ is   left pure semisimple.

(5) $\Rightarrow$ (4). It is clear, since by  [8, Theorem 4.4],    every left pure semisimple  ring is a left Artinian ring.

(4) $\Rightarrow$ (1).  Since $R$ is  a  left Artinian ring,  $R$ is a left prefect ring by [1, Corollary 28.8]. Also, by Corollary \ref{top square-free is cyclic},   $R$ is  a left K{\"o}the ring.  Since every   left $R$-module  is a direct sum of modules with square-free top, we conclude that $R$ is a strongly left  K{\"o}the  ring.

(5) $\Rightarrow$ (6).  By  [8, Theorem 4.4],   $R$ is  a left Artinian ring. Hence $R$ is a left prefect ring. Also, by Corollary \ref{top square-free is cyclic}, $R$ is  a left K{\"o}the ring.  Thus,  $R$ is of finite length and since every  (finitely generated)   indecomposable  left $R$-module is cyclic, $R$ is  of left bounded representation type,   and hence by [32, Proposition 54.3], $R$  is  of finite representation type.

In the next steps of proof, for the finite representation type ring $R$, by  Proposition \ref{equivalence},  we have the following assumptions.

\noindent (i) The left Auslander ring  of $R$ is  $T={\rm End}_R(U)$, where  $U=U_1\oplus\cdots\oplus U_n $ and  $\{U_1,\ldots,U_n\}$ \indent is   a complete set of representatives of the isomorphic classes of finitely
generated inde- \indent composable left $R$-modules.\\
(ii)  The  left Morita dual ring of $R$  is   $S=End(Q)$, where  $Q\cong E(S_1)\oplus\cdots\oplus E(S_m)$, where  \indent $\{S_i~|~1\leq i\leq m\}$ is a complete  set  of representatives of the isomorphic
classes of simple \indent left $R$-modules.\\
(iii)   $\{Hom_R(U_1,Q),\ldots,Hom_R(U_n,Q)\}$ is a complete  set of  representatives of the isomor-  \indent  phic  classes of finitely generated indecomposable  right  $S$-modules.

(6) $\Rightarrow$ (7). Since $\{Hom_R(U_1,Q),\ldots,Hom_R(U_n,Q)\}$ is a complete set of representatives of the isomorphic classes of finitely generated indecomposable right $S$-modules,   and since  each $U_i$ has  a square-free top,    by Proposition \ref{equivalence} (i),  each indecomposable right $S$-module ${\rm Hom}_R(U_i,Q)$   has a square-free socle.  

(7) $\Rightarrow$ (8).  Since $R$  of finite representation type,  by Proposition \ref{equivalence} (d),  the  left Auslander ring  $T$ of $R$ is isomorphic to the  right  Auslander ring  $T^\prime$ of  $S$. Thus,  every   indecomposable projective  right  $T$-module  has a  square-free  socle.

(8) $\Rightarrow$ (6).  By Proposition \ref{equivalence} (d),  $T\cong T^\prime$,  where   $T^\prime$ is the right  Auslander ring  of $S$. By Proposition \ref{equivalence} (e) and (c),  each of $Hom(U_{1} , Q),\ldots, Hom(U_{n} , Q)$ has a square-free socle. By Proposition \ref{equivalence}  (m), we have $soc(Hom(U_{i} , Q)) \cong Hom_{R}(top(U_i), Q)$ and since  each $soc(Hom(U_{i} , Q))$  is  square-free, we conclude that  $top(U_{i})$  is  also square-free for each $i$. Thus $R$ is of finite representation  type and  every  (finitely generated) indecomposable  left $R$-module has a  square-free top.

  (6) $\Rightarrow$ (9). It is by Proposition \ref{Projective cover}(i).

  (9) $\Rightarrow$ (1). It is by Lemma \ref{square-free top}.

    (6) $\Rightarrow$ (10). It is by  Propositions  \ref{Projective cover} and [1, Proposition 27.10] and also by [1, Lemma 27.3]. Also, for $1\leq i\leq n$, every indecomposable module  $Re_i/Ie_i$ has a  square-free top.

    (10) $\Rightarrow$ (1).  Since every indecomposable left $R$-module is  isomorphic to  the cyclic module $Re_i/Ie_i$  for some $1\leq i\leq n$ and  a  left ideal $I\subseteq J$, it is enough to show that $R$ is  of finite representation type. Since $R$ is an Artinian ring, each  $Re_i$ has finite length,  and then  we set $n= max \{length(Re_i)~|~1\leq i\leq n\}$. By assumption every indecomposable module $M$ is a factor of $Re_i$, thus $length(M) \leq n$ and hence $R$ is of bounded representation type. So by [32, Proposition 54.3], $R$ is of finite representation type. Then every left $R$-module is a direct sum of indecomposable cyclic modules and each of indecomposable cyclic module $Re/Ie$ has  a square-free top.
         \end{proof}

 A ring $R$ is said to be {\it left duo} (resp., {\it  left quasi-duo}) if every left ideal (resp. maximal left ideal) of $R$ is an ideal. Obviously, left duo rings are left
quasi-duo. Other examples of left quasi-duo rings include, for instance, the commutative rings, the local rings, the rings in which every non-unit element has a (positive)
power  that  is a  central element, the endomorphism rings of uniserial modules, the  power series rings and  the rings of upper triangular matrices over any of the above-mentioned  rings (see  [33]).  It is easy to see that if a ring $R$ is left duo (resp. left quasi-duo), so is any factor ring of $R$.  By a result of  [33],  a ring $R$ is left quasi-duo if and only if  $R/J(R)$ is left quasi-duo, and  if $R$ is left quasi-duo, then $R/J(R)$ is a subdirect product of division rings. The  converse is not true in general (see [23, Example 5.2]), but   the converse is true when $R$ has only a finite number of simple left $R$-modules up to isomorphisms (see [23, Page 252]).  Consequently,  a semilocal ring $R$ is left  (right) quasi-duo if and only if  $R/J(R)$  is a direct product of division rings. Recall that a ring $R$ is called basic if $R/J$ is a direct sum of division rings and idempotents in $R/J$ can be lifted to $R$. Thus,  any basic ring is a left and a right quasi-duo ring.

The following theorem shows that for any left quasi-duo ring $R$ the concepts of   ``left  K{\"o}the"   and  ``strongly left  K{\"o}the"   coincide.

\vspace*{1mm}

\begin{The}\label{quasi duo left Kothe}
The following statements  are equivalent for a left quasi-duo ring $R$.  \vspace*{0.2cm} \\
\indent {{\rm (1)}}   $R$ is a left  K{\"o}the  ring.\\
\indent {{\rm (2)}}   $R$ is a strongly left  K{\"o}the  ring.\\
\indent {{\rm (3)}}  $R$ is  of finite representation type and  every   indecomposable left $R$-module has \indent\indent a cyclic top.\\
\indent {{\rm (4)}}  $R$ is  of finite representation type and  every   indecomposable left $R$-module has \indent\indent a square-free  top.\\
\indent {{\rm (5)}}  $R$ is  of finite representation  type  and for the  left Morita dual ring of $R$, every  \indent\indent (finitely generated) indecomposable right module has a (essential) cyclic socle.\\
\indent {{\rm (6)}}  $R$ is  of finite representation  type  and for the  left Morita dual ring of $R$, every  \indent\indent (finitely generated) indecomposable right module has a square-free socle.\\
\indent {{\rm (7)}}  $R$ is  of finite representation type  and the left Auslander ring of $R$ is a generalized \indent\indent right $QF$-2 ring.
\end{The}

\begin{proof}
(1) $\Rightarrow$ (2).  Since $R$ is  a left K{\"o}the ring,  By  [8, Theorem 4.4],  $R$ is of finite representation
type and each indecomposable left $R$-module is cyclic. Thus,  we can assume that  $\{R/I_1,\ldots, R/I_n\}$  is a complete set of representatives of the isomorphic classes of finitely generated indecomposable left $R$-modules, where $I_1, \ldots, I_n$  are left ideals of $R$. But $R$ has only a finite number of simple
left $R$-modules (up to isomorphisms)  and   $R$ is left quasi-duo, so by [23, Page 252],   $R/J$  is a finite  direct product of division rings and   for each maximal  (left) ideals $P_1\neq P_2$ of $R$,
$R/P_1\not\cong R/P_2$.  Now by [1, Corollary 15.18],  we have ${\rm Rad}(R/I_i)=JR/I_i$ for each $1\leq i\leq n$. Thus,
$${\rm top}(R/I_i)= \frac{R/I_i}{J(R/I_i)}\cong \frac{R}{J+I_i}\cong  \bigoplus_{I_i\leq {P_j}\in {\rm Max}(R)}\frac{R}{P_j}\cong\frac{R}{\bigcap_{I_i\leq P_j\in {\rm Max}(R)}P_j}.$$
  So,   every   indecomposable left $R$-module has  a square-free top. Hence by Theorem \ref{strongly left Kothe}, the proof is complete.

(2) $\Rightarrow$ (1).  Since $R$ is of finite representation type, we can assume that  $\{U_1, \ldots, U_n\}$  is a complete set of representatives of the isomorphic classes of finitely generated indecomposable left $R$-modules. Since  by our hypothesis each $U_i$  $(1\leq i\leq n)$ has a square-free top, each $U_i$ is cyclic by Corollary \ref{top square-free is cyclic}, and hence every left $R$-module is a direct sum of cyclic modules, i.e., $R$ is a left  K{\"o}the  ring.

 (2) $\Leftrightarrow$ (4)  $\Leftrightarrow$ (6)  $\Leftrightarrow$ (7). By Theorem \ref{strongly left Kothe}.

 (1) $\Leftrightarrow$ (3) $\Leftrightarrow$ (5). By Theorem \ref{left Kothe}.
\end{proof}

 \section{\bf Charactrizations of very strongly left  K{\"o}the rings}
The rings $R$ satisfying  the following $(*)$ condition  were first studied
by Tachikawa [29] in 1959, by using duality theory. 

\vspace*{1mm} 

\noindent {\it $(*)$  $R$   is a  right Artinian  ring and  every finitely generated indecomposable right  $R$-module  is local. } 
\vspace*{1mm} 

\noindent But, Singh and Al-Bleahed  [28] have studied  rings $R$ satisfying  $(*)$ without using duality.
Since each local  module is  cyclic indecomposable and it has  a simple top,   the  condition $(*)$  is  a stronger condition than the  strongly right K{\"o}the  condition. In fact, we will see that the rings with $(*)$ condition are exactly very strongly right K{\"o}the rings. In the first result we give several characterizations of  very strongly left K{\"o}the rings.

\begin{The}\label{left local type}
The following statements are equivalent for any  ring $R$.  \vspace*{0.2cm} \\
\indent {{\rm (1)}}  $R$ is a very strongly left  K{\"o}the  ring.\\
\indent {{\rm (2)}}  Every  nonzero left $R$-module is a direct sum of  modules with nonzero socle and  \indent\indent simple top.\\
\indent {{\rm (3)}}  $R$  is  left  perfect,   and every left  $R$-module is a direct sum of  modules with  simple \indent\indent top.\\
\indent {{\rm (4)}}  $R$ is  left Artinian,   and  every   left $R$-module  is a direct sum of modules with simple \indent\indent top.\\
\indent {{\rm (5)}} $R$ is  left pure semisimple,   and  every   left $R$-module  is a direct sum of modules with \indent\indent simple top.\\
\indent {{\rm (6)}}  $R$ is  of finite representation type,  and  every  $($finitely generated$)$ indecomposable left  \indent\indent $R$-module  has a simple  top.\\
 \indent {{\rm (7)}} Every left $R$-module is a direct sum of lifting modules.\\
 \indent {{\rm (8)}} Every left $R$-module is a direct sum of local modules.\\
\indent {{\rm (9)}}  $R$ is  of finite representation  type,   and for the  left Morita dual ring of $R$
every  \indent\indent (finitely generated) indecomposable right module has a simple socle.\\
\indent {{\rm (10)}}  $R$ is  of finite representation  type,   and the left  Auslander ring  of $R$  is a right \indent\indent $QF$-2 ring.\\
 \indent {{\rm (11)}}   $R$ is of finite representation type,   and   for each  indecomposable left   $R$-modules $U$,
 \indent\indent  $top(U) \cong Re_i/Je_i$  where  $e_i \in \{e_1, \ldots, e_n\}$ and $\{e_1, \ldots, e_n\}$ is a basic set of primi- \indent\indent tive  idempotents for $R$.\\
\indent {{\rm (12)}}   $R$ is  an Artinian  right serial  ring,   $R\cong \oplus_{i=1}^n (Re_i)^{(t_i)}$, where  $\{Re_1, \ldots, Re_n\}$   is a  \indent\indent complete   set of   representatives of the isomorphic classes of
indecomposable projec-  \indent\indent tive left $R$-modules with $n,~t_i\in\Bbb{N}$,  and  each $Re_i$ has a simple  top and  each inde- \indent\indent composable  left $R$-module is  isomorphic to $Re_i/Ie_i$  for some $1\leq i\leq n$
 and  a  left \indent\indent ideal $I\subseteq J$.
\end{The}
\begin{proof}
The equivalence of part (1) to (6) is obtained by Theorem \ref{left Kothe}, since every module with simple top is a module with square-free cyclic top.

 (6) $\Leftrightarrow$ (7). It is by [14, Theorem 2.1].

(1) $\Leftrightarrow$ (8). It is by Proposition \ref{Projective cover} (ii).

In the next steps of proof, for the finite representation type ring $R$, by  Proposition \ref{equivalence},  we have the following assumptions.

\noindent (i) The left Auslander ring  of $R$ is  $T={\rm End}_R(U)$, where  $U=U_1\oplus\cdots\oplus U_n $ and  $\{U_1,  \ldots,U_n\}$ \indent is   a complete set of representatives of the isomorphic classes of finitely
generated inde- \indent composable left $R$-modules.\\
(ii)   The  left Morita dual ring of $R$  is    $S=End(Q)$, where  $Q\cong E(S_1)\oplus\cdots\oplus E(S_m)$, where  \indent $\{S_i~|~1\leq i\leq m\}$ is a complete   set  of representatives of the isomorphic
classes of simple \indent left $R$-modules.\\
(iii)   $\{Hom_R(U_1,Q), \ldots, Hom_R(U_n,Q)\}$ is a complete  set of  representatives of the isomor-  \indent phic classes of finitely generated indecomposable  right  $S$-modules.

(8) $\Rightarrow$  (9).
 By Proposition \ref{equivalence} (c) and (e), it is enough to show that each  of indecomposable right $S$-modules $Hom_R(U_1, Q), \ldots, Hom_R(U_n, Q)$,   has a simple socle.  By our hypothesis,  each  $U_i$  has  a simple top.  By Proposition \ref{equivalence}  (n), we have  $soc(Hom(U_i, Q)) \cong soc(\bigoplus_{k=1}^{m}Hom(S_{i}, Q)) = Hom(top(U_i) , Q)$ is simple. It  follows that  every indecomposable right $S$-module  $Hom(U_i,Q)$   has a simple socle.

 (9) $\Rightarrow$  (10).
By Proposition \ref{equivalence} (d),  $T\cong T^\prime$,  where   $T^\prime$ is the right  Auslander ring  of $S$. Since  each  of indecomposable right $S$-module has a simple socle,  so   every finitely generated indecomposable projective right $T$-module  has a simple socle, i.e., $T$ is  a right  $QF$-2 ring.

 (10) $\Rightarrow$  (6).  By Proposition \ref{equivalence} (d),  $T\cong T^\prime$,  where   $T^\prime$ is the right  Auslander ring  of $S$. By Proposition \ref{equivalence} (e) and (c),  each of $Hom(U_{1} , Q), \ldots, Hom(U_{n} , Q)$ has a simple socle. By Proposition \ref{equivalence}  (m), we have $soc(Hom(U_{i} , Q)) \cong Hom_{R}(top(U_i), Q)$ and since  each $soc(Hom(U_{i} , Q))$  is  simple, we conclude that  $top(U_{i})$  is  also simple for each $i$. Thus,
 $R$ is of finite representation   type and  every  $($finitely generated$)$ indecomposable  left $R$-module has a simple top.

 (6) $\Rightarrow$  (11). By Prposition \ref{Projective cover} (ii)

   (11) $\Rightarrow$  (6). It is clear.

     (8) $\Rightarrow$ (12).   By [28, Theorem 2.4],    $R$ is  an Artinian  right serial  ring.  Also, by [1, Proposition 27.10],    $R\cong \oplus_{i=1}^n(Re_i)^{(t_i)}$ where $n,~t_i\in\Bbb{N}$, $i=1, \ldots, n$ and $\{e_1,\cdots, e_n\}$  is a basic  set   of idempotents of $R$. By Proposition  \ref{Projective cover} (ii),   each $Re_i$ is local for $i=1, \ldots, n$.  Now  let $M$ be a local left $R$-module. Since  any local module is cyclic,  by  [1, Lemma 27.3], $M\cong Re/Ie$ for some idempotent $e\in R$ and some left ideal $I \subseteq J$,  and  the natural map
$Re\rightarrow   Re/Ie\rightarrow  0$  is a projective cover  of $M$. Since $M$ is indecomposable cyclic left $R$-module with simple top, by Propositon  \ref{Projective cover} (ii), $Re$ is  an indecomposable direct summand of $_RR$ and so by [1, Corollary  7.4],
$e$ is a primitive idempotent. It follows that by Proposition [1, Proposition 27.10], $Re\cong Re_i$ for some $i$ ($1\leq i\leq n$) and so $M\cong Re_i/Ie_i$, also $\{Re_1, \ldots, Re_n\}$ is a  (finite) complete   set of   representatives of the isomorphic classes of indecomposable projective left $R$-modules. By Proposition  \ref{Projective cover} (ii), $top(M) = Re_i/Je_i$ is simple  for $1\leq i\leq n$ and proof is complete.

   (12) $\Rightarrow$ (1).  Since every indecomposable left $R$-module is  isomorphic to  the cyclic module $Re_i/Ie_i$  for some $1\leq i\leq n$
and  a  left ideal $I\subseteq J$, it is enough to show that $R$ is  of finite representation type. Since $R$ is an Artinian ring, each  $Re_i$ has finite length. Set $n= max \{length(Re_i)~|~1\leq i\leq n\}$.  By assumption every indecomposable module $M$ is a factor of $Re_i$, thus $length(M) \leq n$ and hence $R$ of bounded representation type. So by [32, Proposition 54.3], $R$ is of finite representation type. Then every left $R$-module is a direct sum of indecomposable cyclic modules with simple top.
 \end{proof}

In the rest of paper, as an application, we combine some results of Singh and Al-Bleahed [28] with our main theorems and obtain some interesting results on local K{\"o}the rings and Abelian K{\"o}the rings. Also we give a new generalization of K{\"o}the-Cohen-Kaplansky theorem. The following lemmas are helpful.

\begin{Lem}\label{Abelian} \textup {([15, Proposition 3])}. Let $R$ be a left Artinian ring. Then $R$ is a finite
product of local rings if and only if $R$ is Abelian.
\end{Lem}

\begin{Lem}\label{Local}{\rm (See [28, Theorem 2.12])}. Let $R$ be a left Artinian  local ring such that every finitely
generated indecomposable left $R$-module is local. Then either $J^2 = 0$ or $R$ is a uniserial
ring.
\end{Lem}

\begin{The}\label{local left Kothe}
The following statements are equivalent for a local ring $R$.  \vspace*{0.2cm} \\
\indent {{\rm (1)}}  $R$ is a (very strongly) left K{\"o}the ring.\\
\indent {{\rm (2)}} $R$ is  of finite  representation type  and the  left  Auslander ring  $T$ of $R$ is a right $QF$-2 \indent\indent ring.\\    
\indent {{\rm (3)}}  Either $R$ is  an Artinian principal ideal ring  or  $R$ is local  with $$J=soc(_RR)=S_1\oplus\cdots\oplus S_n$$     \indent\indent  and $\{R/I_k~|~I_k=S_1\oplus\cdots\oplus S_k, ~ 1\leq k\leq n\}$  is  a mutually non-isomorphic finitely \indent\indent  generated indecomposable left $R$-module. 
\end{The} 
\begin{proof}
It follows by  Theorem \ref{left local type} and [28, Theorem 3.8].
\end{proof}

\begin{Cor}\label{local  Kothe}
The following statements are equivalent for a local ring $R$. \vspace*{0.2cm} \\
\indent {{\rm (1)}}  $R$ is a K{\"o}the ring.\\ 
\indent {{\rm (2)}} $R$ is a strongly K{\"o}the ring.\\
\indent {{\rm (3)}}  $R$ is a very strongly K{\"o}the ring.\\
\indent {{\rm (4)}}  Any Morita equivalent to $R$ is a K{\"o}the ring.\\
\indent {{\rm (5)}} $R$ is  of finite  representation type  and the  left (right)  Auslander ring  $T$ of $R$ is a  \indent\indent $QF$-2 ring.\\ 
\indent {{\rm (6)}} $R$ is  an Artinian principal ideal ring. 
\end{Cor}
\begin{proof}
It follows by Proposition \ref{Kawada}, Theorem \ref{local left Kothe} and [28, Theorem 3.8]. 
\end{proof}

\begin{Pro}\label{Cor prime ideals commutes-Abelian}
The following statements are equivalent for an Abelian ring $R$.  \vspace*{0.2cm} \\
\indent {{\rm (1)}}    Any Morita equivalent to $R$ is  a K{\"o}the ring.\\
\indent {{\rm (2)}} $R$ is  an Artinian principal ideal ring.\\
\indent {{\rm (3)}} $R\cong R_1\times\cdots\times R_k$, where $k,~n_1,\ldots,n_k\in\Bbb{N}$  and $R_i$ is a   local    Artinian   principal \indent\indent ideal   rings for each $1\leq i\leq k$.\\
\indent {{\rm (4)}} $R$ is a K{\"o}the ring.\\
\indent {{\rm (5)}} $R$ is a strongly K{\"o}the ring.\\
\indent {{\rm (6)}} $R$ is a very strongly K{\"o}the ring.\\
\indent {{\rm (7)}} Every left $R$-module is semidistributive.\\
\indent {{\rm (8)}} Every right $R$-module is semidistributive.\\
\indent {{\rm (9)}} Every left $R$-module is serial.\\
\indent {{\rm (10)}} Every right $R$-module is serial.\\
\indent {\rm (11)} $R$ is an Artinian serial ring.\\
\indent {{\rm (12)}} $R$ is isomorphic to a finite product of Artinian uniserial rings.
 \end{Pro}  
\begin{proof} 
(1) $\Leftrightarrow$ (4). This is obtained by Proposition \ref{Kawada} and the fact that  Artinian Abelian rings are basic.

(2) $\Leftrightarrow (3) \Leftrightarrow$ (12). See [7, Corollary 3.3].
 
(3) $\Leftrightarrow$ (4). It is obtained by Lemma \ref{Abelian}, [30, Theorem 3.6] and the fact that in any Artinian local ring the Jacobson radical is the unique prime ideal.

(4) $\Leftrightarrow (5) \Leftrightarrow$ (6). These follow from Lemma \ref{Abelian} and Corollary \ref{local  Kothe}.

(4) $\Leftrightarrow (7) \Leftrightarrow$ (8). These are ontained by [31, Theorem 11.6].

$(9) \Leftrightarrow (10) \Leftrightarrow (11)$. These follow from [32, Proposition 55.16].

$(12) \Rightarrow (11)$ is clear and $(11) \Rightarrow (2)$ is obtained by [32, Proposition 55.1(2)].
\end{proof}  

An Artinian ring $R$ is said to have a {\it self}-({\it Morita}) {\it duality} if there is a Morita duality $D$ between $R$-mod, the category of finitely generated left $R$-modules, and mod-$R$, the category of finitely generated right $R$-modules. Since $R$ is Artinian, by what Morita [24] and Azumaya [5] have shown:
$R$ has a self-duality $D$ if and only if there is an injective cogenerator $_RE$ of $R$-mod and a ring isomorphism $v: R \rightarrow End(_RE)$ (which induces a right $R$-structure on $E$ via $x . r = xv(r)$ for $x \in E$ and $r \in R$) such that the dualities $D$ and  $Hom_R( -, _RE_R)$ are naturally equivalent. In the last resut we give more equivalent conditions for a very strongly K{\"o}the ring. In particular, we show that $R$ is a very strongly K{\"o}the ring if and only if $R$ is an Artinian serial ring. This equivalence is a generalization of Theorem \ref{abelian Kothe rings}, and so it is a new generalization of K{\"o}the-Cohen-Kaplansky theorem.
 
 \begin{The}\label{co-Kothe}
The following conditions are equivalent for a ring  $R$. \vspace*{0.2cm} \\
\indent {{\rm (1)}} $R$ is a very strongly K{\"o}the ring.  \\
\indent {{\rm (2)}} $R$ is an Artinian serial ring.\\
\indent {{\rm (3)}}  $R$ is of finite representation  type and has colocal type (i.e., every finitely generated \indent\indent indecomposable left $R$-module has a simple socle).\\
\indent {{\rm (4)}} Every left and right $R$-module is a direct sum of uniform modules.\\
\indent {{\rm (5)}} Every left and right $R$-module is a direct sum of extending modules.\\
\indent {{\rm (6)}} $R$ is   of finite  representation type and the left  (right) Auslander ring  of $R$ is a  $QF$-2 \indent \indent ring.\\
\indent {{\rm (8)}} Every left $R$-module is a direct sum of  (finitely generated)  indecomposable  modules \indent  \indent with   simple  socle and simple top. 
\end{The}
\begin{proof}
The equivalence (2) $\Leftrightarrow$ (5) follows from [10, Corollary  2], and the others are obtained by Theorem \ref{left local type} and the fact that  every Artinian serial ring  has self-duality.
\end{proof}

 \end{document}